\documentclass[12pt,reqno]{amsart}
\usepackage{amssymb,mathrsfs,centernot}
\usepackage{color}

\newcommand{\R}{\mathbb{R}}

\newcommand{\al}{\alpha}

\newcommand{\de}{\delta}
\newcommand{\e}{\varepsilon}
\newcommand{\fy}{\varphi}

\newcommand{\te}{\theta}
\newcommand{\s}{\sigma}

\newcommand{\y}{\eta}
\newcommand{\z}{\zeta}

\newcommand{\ro}{\rho}

\newcommand{\De}{\Delta}

\newcommand{\p}{\partial}

\newcommand{\Ca}{\bigcap}

\newcommand{\lec}{\lesssim}

\newcommand{\I}{\infty}

\newcommand{\LR}[1]{{\langle #1 \rangle}}

\newcommand{\EQ}[1]{\begin{equation}\begin{split} #1 \end{split}\end{equation}}

\newcommand{\BR}[1]{\left[#1\right]}

\newcommand{\Del}[1]{}
\newcommand{\CAS}[1]{\begin{cases} #1 \end{cases}}

\newcommand{\pt}{&}
\newcommand{\pr}{\\ &}
\newcommand{\pq}{\quad}

\newcommand{\pn}{}

\numberwithin{equation}{section}

\newtheorem{thm}{Theorem}[section]

\newtheorem{lem}[thm]{Lemma}
\newtheorem{prop}[thm]{Proposition}

\theoremstyle{remark}
\newtheorem{rem}{Remark}

\begin{document}
\title{Non-uniqueness for an energy-critical heat equation on $\R^2$}
\author{Slim Ibrahim, Hiroaki Kikuchi, Kenji Nakanishi, Juncheng Wei}
\begin{abstract}
We construct a singular solution of 
a stationary nonlinear Schr\"{o}dinger equation on $\R^2$ with square-exponential nonlinearity having linear behavior around zero. 
In view of Trudinger-Moser inequality, this type of nonlinearity has an energy-critical growth. We use this singular solution to prove non-uniqueness of strong solutions for the Cauchy problem of the corresponding semilinear heat equation. 
The proof relies on explicit computation showing a regularizing effect of the heat equation in an appropriate functional space.  
\end{abstract}
\maketitle
\section{Introduction}

In this paper, we consider the Cauchy problem 
for the following semilinear heat equation 
\EQ{ \label{eq-Heat}
\begin{cases}
\dot  u - \Delta u = f(u) & 
\qquad \mbox{in $(0, \infty) \times \mathbb{R}^{d}$}, \\
u(0, x) = u_{0}(x)& \qquad \mbox{in $\mathbb{R}^{d}$}, 
\end{cases}}
where $u(t, x) : (0, \infty) \times \mathbb{R}^{d} \to \mathbb{R}$, $d\geq2$
is the unknown function, $f$ is the nonlinearity and $u_{0}\in L^q(\R^d)$ with $1\leq q\leq\infty$ is the given initial data. 

It is well known that when $f$ is $C^1(\R,\R)$, the Cauchy problem  \eqref{eq-Heat} possesses a unique classical solution if the initial data $u_0\in L^\infty(\R^d)$.
For unbounded initial data, this equation has been studied intensively since the pioneering
work of Weissler~\cite{Weissler1}. For instance, in the pure power case i.e. $f(u) = |u|^{p-1}u\; (p>1)$, equation \eqref{eq-Heat} becomes scale invariant, that is, if $u(t, x)$ satisfies \eqref{eq-Heat}, 
then so does  
\EQ{
u_{\lambda}(t, x) = \lambda^{\frac{2}{p-1}}u(\lambda^{2}t, \lambda x)
}
for any $\lambda>0$. 
Moreover, if one defines the index
\EQ{
q_{c} = \frac{d(p-1)}{2}, 
}
then the Lebesgue space $L^{q_{c}}(\mathbb R^d)$ becomes invariant and 
we have $\|u_{\lambda}\|_{L^{q_{c}}} = \|u\|_{L^{q_{c}}}$ for all $\lambda>0$. 
The {\it critical} exponent $q_{c}$ then plays an important role 
for the well-posedness of the Cauchy problem. Indeed, first recall Weissler's  result ~\cite{Weissler1} concerning the 
subcritical case $q > q_{c}$ or the critical case $q=q_c>1$. 
Weissler proved that for any $u_{0} \in L^{q}(\mathbb{R}^{d})$, 
there exists a local time $T = T(u_{0})>0$ and 
a solution $u \in C([0, T); L^{q}(\mathbb{R}^{d})) 
\cap L_{\text{loc}}^{\infty}((0, T); L^{\infty}(\mathbb{R}^{d}))$ to 
\eqref{eq-Heat}. 
After that,  Brezis and Cazenave~\cite{Brezis-Cazenave} proved the {\it unconditional} uniqueness of 
Weissler's solutions i.e. solution is unique in $ C([0, T); L^{q}(\mathbb{R}^{d}))$ when the subcritical case $q>q_c$, $q\geq p$ or 
the critical case $q=q_c>p$.  In the {\it supercritical} case $q < q_{c}$, 
Weissler~\cite{Weissler1}, and Brezis and Cazenave~\cite{Brezis-Cazenave} 
indicated that, for a specific initial data, there is no local solution in any reasonable weak sense. 
Moreover, Haraux and Weissler~\cite{Haraux-Weissler} 
proved non-uniqueness of the trivial solution 
in $C([0, T); L^{q}(\mathbb{R}^{d})) \cap L_{\text{loc}}^{\infty}
((0, T); L^{\infty}(\mathbb{R}^{d}))$  when $1 + 2/d < p < (d+2)/(d-2)$. 

\par
In the critical case $q = q_{c}$ and $d \geq 3$, when $q_{c} > p-1$,  
Weissler~\cite{Weissler1} proved the existence of solutions
to \eqref{eq-Heat} and 
Brezis and Cazenave~\cite{Brezis-Cazenave} obtained the {\it unconditional} uniqueness of the solutions. 
In the double critical case of $q = q_{c} = (p-1) (= d/(d-2))$, 
Weissler~\cite{Weissler2} proved the conditional well-posedness (uniqueness in a subspace of $C([0, T); L^{q}(\mathbb{R}^{d}))$). 
In the case where the underlying space is the ball of $\mathbb{R}^{d}$ 
with Dirichlet boundary condition,  
Ni and Sacks~\cite{Ni-Sacks} showed that the unconditional 
uniqueness fails, while Terraneo~\cite{Terraneo} and Matos and Terraneo~\cite{Matos-Terraneo} 
extended this result to the entire space $\mathbb{R}^{d} \; (d \geq 3)$. 
\par
We note that the critical exponent $q_{c}$ is also important 
for the blow-up problem \eqref{eq-Heat}. 
Let $u_{0} \in L^{\infty}(\mathbb{R}^{d})$ and 
$T(u_{0})$ be the maximal existence of the time of the classical 
solution $u$. 
It is known that if $T(u_{0}) < \infty$, the solution $u$ satisfies 
$\lim_{t \to T} \|u(t)\|_{L^{\infty}} = \infty$ and 
we say that the solution $u$ blows up in finite time and 
$T(u_{0})$ is called the blow-up time of $u$. 
In particular, the {\it critical $L^{q_{c}}$ norm blow-up problem} 
has attracted attention for a long time. 
Namely, it is a question whether the solution 
satisfies 
\begin{equation}
\sup_{t \in (0, T)} \|u\|_{L^{q_{c}}} = \infty
\end{equation}
or not when $T(u_{0}) < \infty$. 
Concerning this problem, 
Mizoguchi and Souplet~\cite{Mizoguchi-Souplet} 
recently showed that 
if $u$ is a type I blow-up solution of \eqref{eq-Heat}, that is, 
$\limsup_{t \to T} (T - t)^{1/(p-1)} \|u(t)\|_{L^{\infty}} < \infty$, 
then we have $\lim_{t \to T} \|u(t)\|_{L^{q_{c}}} = \infty$. 
\par
Now, let us pay our attention to the two dimensional case. 
When we consider the two space dimensions, 
we see that any exponent $1 < p < \infty$ is subcritical, and thanks to the result of Weissler~\cite{Weissler1}, we have the local well-posedness of the Cauchy  problem in $L^{q}(\mathbb{R}^{2})\; (1 < p<q < \infty)$. 
%
%
%
%
%
%
%
%
%
However, for exponential type nonlinearities, like 
$f(u) = \pm u(e^{u^{2}} -1)$, 
the result of Weissler~\cite{Weissler1} is not 
applicable for any Lebesgue space $L^{q}(\mathbb{R}^{2}) (1 < q < \infty)$. 
On the other hand, we can show the local well-posedness for 
$u_{0} \in L^{\infty}(\mathbb{R}^{2})$, as we mentioned first.
However, $L^{\infty}(\mathbb{R}^{2})$ is the subcritical space 
for the problem \eqref{eq-Heat} with $d=2$ and exponential type 
nonlinearities. 
Therefore, one can wonder if there is any notion of criticality 
in two space dimensions.
In this regard, Ibrahim, Jrad, Majdoub and Saanouni~\cite{IJMS}
have considered the following problem in two space dimensions, 
\EQ{ \label{eq-Heat-expo}
\begin{cases}
\dot  u - \Delta u =f_{0}(u) := \pm u(e^{u^{2}} -1)& 
\qquad \mbox{in $(0, \infty) \times \mathbb{R}^{2}$}, \\
u(0, x) = u_{0}(x)& \qquad \mbox{in $\mathbb{R}^{2}$}.  
\end{cases}}
The nonlinearity $f_{0}(u)$ has an energy-critical growth in view of Trudinger-Moser inequality. 
In \cite{IJMS}, the authors proved the local existence and uniqueness 
in $C([0,T],H^1(\mathbb{R}^{2}))$ of the solution to \eqref{eq-Heat-expo} with the 
initial data $u_{0} \in H^{1}(\mathbb{R}^{2})$. 

On the other hand, it is expected that the problem \eqref{eq-Heat-expo} 
for the heat equation can be
solved in spaces which are defined by an integrability of functions such as the Orlicz space as an
extension of the class of Lebesgue spaces. 
Ruf and Terraneo~\cite{Ruf-Terraneo}
showed the local existence of a solution 
to \eqref{eq-Heat-expo} for small initial data in 
the Orlicz space $\text{exp}L^2$ defined by
$$
\text{exp}L^2(\R^2):=\{u\in L^1_{loc}(\R^2):\;\text{such that}\; \int_{\R^2}\exp( u^2/\lambda^2)-1\;dx<\infty, \;\text{for some}\;\lambda>0\}
$$ 
endowed with Luxemburg norm
$$
\|u\|_{\text{exp}L^2}:=\inf\{\lambda>0: \int_{\R^2}\exp( u^2/\lambda^2)-1\;dx\leq 1\}.
$$

Subsequently, Ioku~\cite{Ioku} has shown that these local solutions are indeed global-in time, and the behavior of $f_0(u)\sim u^m$ near $u\sim0$ with $m\geq3$ was important in his result. 
Later on, 
Ioku, Ruf and Terraneo~\cite{IRT} proved that local solutions do not exist for the initial data
\EQ{
\label{non exi ID}
u_*(r)=a (-\log r)^\frac12,\;\text{for}\; r\leq1,\;\; a\gg 1,\;\text{and}\; u_*(r)=0,\;\text{elswhere}
}
that belongs to the Orlicz space, while they showed the local 
well-posedness (local existence and uniqueness) 
for initial data in the subclass of the Orlicz space
$$
\text{exp}L^2_0(\R^2):=\{u\in L^1_{loc}(\R^2):\;\text{such that}\; \int_{\R^2}\exp(\alpha u^2)-1\;dx<\infty, \;\text{for every}\;\alpha>0\}
$$ 
\par
The aim of this paper is to 
construct an explicit initial data, with neither small nor too large norm in Orlicz space, for which two corresponding distinct solutions exist. The idea is to use the concept of {\it singular solutions} as in Ni and Sacks \cite{Ni-Sacks}.
Before stating our result, let us recall the strategy 
of the proof of \cite{Ni-Sacks}. 
They first construct a singular static solution 
$\phi_{*}$ to \eqref{eq-Heat} in the unit ball. 
Namely, $\phi_{*}$ satisfies 
following:
 \EQ{ \label{eq-static}
\begin{cases}
- \Delta \phi = f(\phi) & \qquad \mbox{in $B_{1}$}, \\
\phi = 0 \quad \mbox{on $\partial B_{1}$}, &\quad 
\lim_{x \to 0} \phi(x) = \infty,  
\end{cases}}
where $B_{1}$ is the unit ball in 
$\mathbb{R}^{d}\; (d \geq 3)$ and $f(\phi) = |\phi|^{p-1}\phi$. 
Then, they showed that there exists a regular solution 
$u_{\text{R}}$ to the Dirichlet problem corresponding to \eqref{eq-Heat} with $u(0, x) = \phi_{*}$. 
Setting $u_{\text{S}} = \phi_{\text{S}}$, 
we see that both of 
$u_{\text{S}}$ and $u_{\text{R}}$ are solutions to \eqref{eq-Heat}, 
but $u_{\text{S}} \neq u_{\text{R}}$
because $u_{\text{R}} \in L^{\infty}(B_{1})$ 
while $u_{\text{S}} \notin L^{\infty}(B_{1})$
(the subscripts $S$ and $R$ indicating `singular' and `regular' solutions). 
For the entire space $\mathbb{R}^{d}$, 
Terraneo~\cite{Terraneo}
constructed a solution $\phi_{*} \in C^{2}(\mathbb{R}^{d} \setminus 
\{-x_{0}, x_{0}\})$ to the equation
\begin{equation} \label{eq-static-entire}
\begin{cases}
- \Delta \phi =  f(\phi) & \qquad \mbox{in $\mathbb{R}^{d}$}, \\
\lim_{|x| \to \infty}\phi(x) = 0 &\quad 
\end{cases}
\end{equation}
such that
$\limsup_{x \to x_{0}} \phi(x) = \infty$ and 
$\liminf_{x \to -x_{0}} \phi(x) = -\infty$, 
where $d \geq 3, x_{0} = (1, \ldots, 1)$ and $f(\phi) = |\phi|^{p-1} \phi$. 
However, we cannot apply the method of \cite{Ni-Sacks} nor of \cite{Terraneo}
to two dimensional entire space $\mathbb{R}^{2}$ 
for $f_{0}(u) = (e^{u^{2}} - 1)u$ directly. 
Indeed, it seems that there is no nontrivial solution to 
the equation \eqref{eq-static-entire} with $f(u) = (e^{u^{2}} - 1)u$ 
which decays at infinity 
because it is two dimensional massless equations.  
For this reason, we consider 
\EQ{ \label{eq-Heat-exp-m}
\CAS{
\dot  u - \Delta u = f_{m}(u)
& \mbox{in $(0, \infty) \times \mathbb{R}^{2}$}, \\
u(0, x) = u_{0}(x) & \mbox{in $\mathbb{R}^{2}$},
}}
where 
our nonlinearity $f_{m}$, which depends on $m >0$, satisfies 
\EQ{ \label{asy-non}
\lim_{u\to\infty} \frac{f_m(u)}{(e^{u^2}-1)u} = 1, \qquad 
  \lim_{u\to 0} \frac{f_m(u)}{m u} = -1. 
}
See \eqref{expo-nonlinear} below for the precise form of $f_{m}$. 
Here, we would like to stress that the essential characterization 
of the asymptotic form of our nonlinearity $f_{m}$ at infinity and $0$ 
is just given by \eqref{asy-non}. 
Let $X$ be the Fr\'echet space defined as the intersection of Banach spaces
\EQ{
 X:=\Ca_{1\le p<\I} L^p(\R^2).} 
Then, our main result is the following: 
\begin{thm} \label{main-thm}
There exist a positive mass $m_{*}>0$, 
a nonlinearity $f_{m_*}$ satisfying \eqref{asy-non}, 
an initial condition 
$\varphi_{*} \in X$ and a time $T = T(\varphi_{*}) >0$ 
such that the Cauchy problem \eqref{eq-Heat-exp-m}  
has two distinct strong solutions $u_{S}$ and 
$u_{R}$ in the sense of Duhamel formula: 
\EQ{
u(t) = e^{t \Delta} \varphi_{*} 
+ \int_{0}^{t} e^{(t-s)\Delta}f_{m_{*}}(u(s))ds \qquad
\mbox{in $C([0, T), X)$}.   
} 
\end{thm}
\begin{rem}
\begin{enumerate}
\item[\rm (i)]
To prove Theorem \ref{main-thm}, we construct a singular stationary 
solution $\varphi_{*}$ to \eqref{eq-Heat-exp-m}. 
Here, by a singular stationary solution, we mean a 
time independent function
which satisfies the equation \eqref{eq-Heat-exp-m} 
in the distribution sense on the whole domain and
diverges at some points.
The result seems to be of independent interest.  
\item[\rm (ii)] 
We essentially use the condition $\lim_{u \to 0}f_{m}(u)/mu = -1$ 
only for the decay of 
a singular stationary soliton $\varphi_{*}$ to \eqref{eq-Heat-exp-m}, 
that is, $\lim_{|x| \to \infty} \varphi_{*}(x) = 0$. 
The other argument in our proof works even without the condition. 
It is a challenging problem to study whether non-uniqueness still 
hold without the condition or not. 
\end{enumerate}
\end{rem}
\begin{rem}
The problem of uniqueness of solutions for PDEs is a classical and old issue that can be tricky sometimes. It has caught a special attention in the last few years. 
Among many others, one can refer to the works of 
De Lellis and Sz\'ekelyhidi~\cite{DS} showing non-uniqueness 
of very weak solutions to Euler problem. 
Their proof relies on convex integration techniques, which  more recently, have been upgraded to show non-uniqueness of weak solutions of Navier-Stokes system thanks to the pioneer work of Buckmaster and Vicol \cite{BV}. 
Davila, Del Pino and Wei~\cite{DDW} constructed non-unique weak
solutions for the two-dimensional harmonic map, by attaching reverse bubbling.
In \cite{GGM}, Germain, Ghoul and Miura investigated the genericity of the non-unique solutions of the supercritical heat flow map. 
\end{rem}

This paper is organized as follows. 
In Section \ref{singular-soliton}, 
we construct a singular static soliton 
$\varphi_{*}$ to \eqref{eq-Heat-exp-m}. 
To this end, we first prove the existence of a singular soliton 
$\phi_{*}$ to 
\eqref{eq-static} with $f(\phi) = e^{\phi^{2}}(\phi - 1)$ 
in the ball in $\mathbb{R}^{2}$, 
following Merle and Peletier~\cite{Merle-Peletier}. 
Then, we seek the singular stationary soliton $\varphi_{*}$ 
to \eqref{eq-Heat-exp-m} by the shooting method. 
In Section \ref{regular-sol}, we shall show the existence of a regular solution to 
\eqref{eq-Heat-exp-m} with $u|_{t=0} = \varphi_{*}$ 
by the heat iteration and give a proof of Theorem \ref{main-thm}. 
In Appendix, we show a monotonicity property of solution to the linear 
heat equation. 

\section{Construction of singular soliton} \label{singular-soliton}
\subsection{Singular stationary solution on some disk}
In this subsection, we construct a singular solution 
to the following elliptic equation on a disk
\EQ{\label{eq-expo}
\begin{cases}
- \Delta u = u(e^{u^{2}} -1 ) \qquad 
& \mbox{in $B_{R}$}, \\
u = 0 \qquad & \mbox{on $\partial B_{R}$},  
\end{cases}
} 
where $R > 0$. 
More precisely, we shall show the following: 
\begin{thm} \label{thm-singular-disk}
There exist $R_{\infty}>0$ and a singular solution 
$u_{\infty} \in C^{\infty}(B_{R_{\infty}} \setminus \{0\})$ 
to \eqref{eq-expo} with $R = R_{\infty}$ satisfying 
\EQ{ \label{sing}
\begin{split}
u_{\infty}(x) 
= \left(-2 \log |x| -2 \log (-\log |x|) -2 \log 2\right)^{\frac{1}{2}} 
 + O((-\log |x|)^{- \frac{3}{2}} 
 \log & (-\log |x|)) \qquad \\ 
&\mbox{as $x \to 0$}
\end{split}
}
\end{thm}
To prove Theorem \ref{thm-singular-disk}, 
we pay our attention for $0 < r = |x| \ll 1$ and 
employ the following 
Emden-Fowler transformation:
\EQ{
y(\ro) = u(x), \qquad \ro = 2 |\log r|. 
} 
Then, we see that the equation in
\eqref{eq-expo} is equivalent to the following: 
\EQ{ \label{ODE}
- \frac{d^{2} y}{d \ro^{2}}  = \frac{e^{-\ro}}{4} y(e^{y^{2}} - 1). 
}
We shall show the following: 
\begin{prop} \label{local-singular}
There exists $\Lambda_{\infty} >0$ and 
a solution $y_{\infty}(\ro)$ to \eqref{ODE} for $\ro \in [\Lambda_{\infty}, \infty)$ 
satisfying 
\EQ{ \label{Asymptotic-local}
y_{\infty}(\ro) = (\ro - 2 \log \ro)^{\frac{1}{2}} 
+ O(\ro^{-\frac{3}{2}} \log \ro) \qquad \mbox{as $\ro \to \infty$}. 
}
\end{prop}
Putting
\EQ{ \label{def-phi}
y(\ro) = \phi(\ro) + \eta(\ro), \qquad 
\phi(\ro) := (\ro - 2\log \ro)^{\frac{1}{2}}, 
}
we have
\EQ{ \label{eq1:singular}
\frac{d^{2} y}{d \ro^{2}} = 
- \frac{\ro^{-\frac{3}{2}}}{4} - \frac{1}{4}(\phi^{-3} 
- \ro^{-\frac{3}{2}})
+ \phi^{- 3}\left\{\frac{1}{\ro} - \frac{1}{\ro^{2}}\right\}
+ \phi^{-1}\ro^{-2} + \frac{d^{2} \eta}{d \ro^{2}}. 
}
Since 
\EQ{
\pt y = \ro^{\frac{1}{2}}\left(1 - 2\frac{\log \ro}{\ro}\right)^{\frac{1}{2}} 
+ \eta
= \ro^{\frac{1}{2}} + \ro^{\frac{1}{2}}
\left\{(1 - 2 \frac{\log \ro}{\ro})^{\frac{1}{2}} - 1\right\} + \eta, 
\pr e^{y^{2} -\ro} 
= \ro^{-2} e^{y^{2} - \phi^{2}}
=\ro^{-2} \left\{ 
1 + (y^{2} - \phi^{2}) + 
\left(e^{y^{2} - \phi^{2}} - 1 - (y^{2} - \phi^{2})\right)\right\}, 
}
we obtain 
\EQ{
\begin{split}
y e^{y^{2}- \ro}
& =  \ro^{-2}\left\{ 
\ro^{\frac{1}{2}} + \ro^{\frac{1}{2}}
\left\{(1 - 2 \frac{\log \ro}{\ro})^{\frac{1}{2}} - 1\right\} 
+ \eta
\right\}  \times \\
& \times \left\{ 
1 + (y^{2} - \phi^{2}) + 
\left(e^{y^{2} - \phi^{2}} - 1 - (y^{2} - \phi^{2})\right)\right\} \\
& = 
\ro^{-2}
\left\{ 
\ro^{\frac{1}{2}} + \ro^{\frac{1}{2}}
\left\{(1 - 2 \frac{\log \ro}{\ro})^{\frac{1}{2}} - 1\right\} 
+ \eta
\right\} \\
& + 
\ro^{-2}
\left\{ 
(y^{2} - \phi^{2}) \ro^{\frac{1}{2}} + (y^{2} - \phi^{2}) \ro^{\frac{1}{2}}
\left\{(1 - 2 \frac{\log \ro}{\ro})^{\frac{1}{2}} - 1\right\} 
+ (y^{2} - \phi^{2}) \eta
\right\} 
\\
& + 
\ro^{-2}
\left(e^{y^{2} - \phi^{2}} - 1 - (y^{2} - \phi^{2})\right)
\left\{ 
\ro^{\frac{1}{2}}(1 - 2 \frac{\log \ro}{\ro})^{\frac{1}{2}} 
+ \eta
\right\} \\
& = 
\ro^{-\frac{3}{2}} + \ro^{-\frac{3}{2}}
\left\{(1 - 2 \frac{\log \ro}{\ro})^{\frac{1}{2}} - 1\right\} 
+ \ro^{-2}\eta \\
& + 
\ro^{- \frac{3}{2}} (y^{2} - \phi^{2}) + (y^{2} - \phi^{2}) \ro^{-\frac{3}{2}}
\left\{(1 - 2 \frac{\log \ro}{\ro})^{\frac{1}{2}} - 1\right\} 
+ \ro^{-2}(y^{2} - \phi^{2}) \eta
\\
& + \ro^{-2}
\left(e^{y^{2} - \phi^{2}} - 1 - (y^{2} - \phi^{2})\right)
\left\{ 
\ro^{\frac{1}{2}}(1 - 2 \frac{\log \ro}{\ro})^{\frac{1}{2}} 
+ \eta
\right\}. 
\end{split}
}
Since
$
y^{2} - \phi^{2} 
= 2\phi \eta + \eta^{2} 
= 2 \ro^{\frac{1}{2}} \eta + \eta^{2} 
+ 2(\phi - \ro^{\frac{1}{2}}) \eta,  
$
we have 
\EQ{
(y^{2} - \phi^{2})\ro^{-\frac{3}{2}}  
= \frac{2 \eta}{\ro} + \ro^{-\frac{3}{2}}\eta^{2} 
+ 2 \ro^{-\frac{3}{2}}(\phi - \ro^{\frac{1}{2}})\eta. 
}
This yields that 
\EQ{ \label{eq2:singular}
\begin{split}
y e^{y^{2}-\ro} 
& = 
\ro^{-\frac{3}{2}} + \ro^{-\frac{3}{2}}
\left\{(1 - 2 \frac{\log \ro}{\ro})^{\frac{1}{2}} - 1\right\} 
+ \ro^{-2}\eta \\
& + 
 2 \frac{\eta}{\ro} + \ro^{-\frac{3}{2}}\eta^{2} 
+ 2 \ro^{-\frac{3}{2}}(\phi - \ro^{\frac{1}{2}})\eta 
\\
& + (y^{2} - \phi^{2}) \ro^{-\frac{3}{2}}
\left\{(1 - 2 \frac{\log \ro}{\ro})^{\frac{1}{2}} - 1\right\} 
+ \ro^{-2}(y^{2} - \phi^{2}) \eta
\\
& + 
\ro^{-2}
\left(e^{y^{2} - \phi^{2}} - 1 - (y^{2} - \phi^{2})\right)
\left\{ 
\ro^{\frac{1}{2}}(1 - 2 \frac{\log \ro}{\ro})^{\frac{1}{2}} 
+ \eta
\right\}. 
\end{split}
}
From \eqref{ODE}, \eqref{eq1:singular} and \eqref{eq2:singular}, we have 
\EQ{
\begin{split}
& \frac{\ro^{-\frac{3}{2}}}{4} 
+ \frac{1}{4}(\phi^{-3} - \ro^{-\frac{3}{2}})
- \phi^{- 3} \left\{\frac{1}{\ro} - \frac{1}{\ro^{2}}\right\}
- \phi^{-1} \ro^{-2} - \frac{d^{2} \eta}{d \ro^{2}}\\
& = 
\frac{\ro^{-\frac{3}{2}}}{4} + \frac{\ro^{-\frac{3}{2}}}{4}
\left\{(1 - 2 \frac{\log \ro}{\ro})^{\frac{1}{2}} - 1\right\} 
+ \frac{\ro^{-2}}{4} \eta \\
& + 
\frac{1}{2 \ro}\eta + \frac{\ro^{-\frac{3}{2}}\eta^{2}}{4} 
+ \frac{1}{2} \ro^{-\frac{3}{2}}(\phi - \ro^{\frac{1}{2}})\eta 
\\
& + \frac{1}{4} (y^{2} - \phi^{2}) \ro^{-\frac{3}{2}}
\left\{(1 - 2 \frac{\log \ro}{\ro})^{\frac{1}{2}} - 1\right\} 
+ \frac{1}{4} \ro^{-2}(y^{2} - \phi^{2}) \eta
\\
& + 
\frac{\ro^{-2}}{4} 
\left(e^{y^{2} - \phi^{2}} - 1 - (y^{2} - \phi^{2})\right)
\left\{ 
\ro^{\frac{1}{2}}(1 - 2 \frac{\log \ro}{\ro})^{\frac{1}{2}} 
+ \eta
\right\} - \frac{e^{-\ro} y}{4}. 
\end{split}
}
Namely, $\eta$ satisfies the following:
\EQ{ \label{eq-eta}
\begin{split}
0  = 
\frac{d^{2} \eta}{d \ro^{2}}
+ \frac{1}{2\ro} \eta + f(\ro) + 
\sum_{i=1}^{6}g_{i}(\ro, \eta) - \frac{e^{-\ro} y}{4},  
\end{split}
}
where 
\begin{align}
& f(\ro) 
= 
- \frac{1}{4}(\phi^{-3} - \ro^{-\frac{3}{2}})
+ \phi^{- 3} \left\{\frac{1}{\ro} - \frac{1}{\ro^{2}}\right\}
+ \phi^{-1} \ro^{-2} 
+ \frac{\ro^{-\frac{3}{2}}}{4}
\left\{(1 - 2 \frac{\log \ro}{\ro})^{\frac{1}{2}} - 1\right\}, \\ 
& g_{1}(\ro, \eta) = \frac{\ro^{-2}}{4} \eta, \quad
g_{2}(\ro, \eta) = \frac{\ro^{- \frac{3}{2}}}{2}(\phi - \ro^{\frac{1}{2}})\eta, \quad 
g_{3}(\ro, \eta) = \frac{\ro^{- \frac{3}{2}}}{4} \eta^{2}, \\
& g_{4}(\ro, \eta) = \frac{1}{4}(y^{2} - \phi^{2}) \ro^{-\frac{3}{2}}
\left\{(1 - 2 \frac{\log \ro}{\ro})^{\frac{1}{2}} - 1\right\}, 
\quad  g_{5}(\ro, \eta) = \frac{\ro^{-2}}{4} (y^{2} - \phi^{2}) \eta, \\  
& g_{6}(\ro, \eta) = \frac{\ro^{-2}}{4}
\left(e^{y^{2} - \phi^{2}} - 1 - (y^{2} - \phi^{2})\right)
\left\{ 
\ro^{\frac{1}{2}}(1 - 2\frac{\log \ro}{\ro})^{\frac{1}{2}} 
+ \eta
\right\}.
\end{align}
To solve the equation \eqref{eq-eta}, 
we first consider the following linear equation:~\footnote{
\eqref{linear} is obtained by removing the 
better decaying terms in right hand of \eqref{eq-eta} 
and by adding $3/(16 \rho^{2})$ in the 
linear potential. 
By adding the term, we can write the solution to \eqref{linear} just
by the trigonometric functions. 
Otherwise, we need to use the 
Bessel functions of order $1$. 
Thus, adding the extra term makes the proof a bit simpler.} 
\EQ{ \label{linear}
\frac{d^{2} \eta}{d \ro^{2}} + \left(\frac{1}{2 \ro} + \frac{3}{16 \ro^{2}} \right) \eta 
= 0. 
}
Any solution $\eta$ to \eqref{linear} can be written explicitly 
as follows:  
we have
\EQ{
\eta (\ro) = A \Phi(\rho) 
+ B \Psi(\rho) 
}
for some $A, B \in \mathbb{R}$, where 
\EQ{
\Phi(\ro) = \ro^{\frac{1}{4}} \sin ((2 \ro)^{\frac{1}{2}}), \qquad 
\Psi(\ro) = \ro^{\frac{1}{4}} \cos ((2 \ro)^{\frac{1}{2}}).   
}
Namely, $\Phi$ and $\Psi$ are the fundamental 
system of solutions to 
\eqref{linear}. 
For a given function $F$, 
we seek a solution to the following 
problem: 
\EQ{ \label{eq-inf}
\begin{cases}
\frac{d^{2} \eta}{d \ro^{2}} + (\frac{1}{2 \ro} + \frac{3}{16 \ro^{2}})
\eta + F = 0, & \qquad \ro \gg 1, \\
\lim_{\ro \to \infty} \eta(\ro) = 0. 
\end{cases}
}
By the variations of parameters, we see that 
\eqref{eq-inf} is equivalent to the following 
integral equation: 
\EQ{
\eta(\ro) = \int_{\ro}^{\infty} 
(\Phi(s) \Psi(\ro) - \Phi(\ro) \Psi(s))
F(s) ds  = \int_{\ro}^{\infty} 
(\ro s)^{\frac{1}{4}} \sin((2\ro)^{\frac{1}{2}} - (2s)^{\frac{1}{2}})
F(s) ds. 
}
By integrating by parts, we can obtain the following:
\begin{lem} \label{lem-sin}
Let $\sigma > 1$. Then, we have
\begin{align}
& \int_{\ro}^{\infty}\sin ((2 \ro)^{\frac{1}{2}} - (2 s)^{\frac{1}{2}})
s^{-\sigma}\log s ds  
= - \sqrt{2} \ro^{-\sigma + \frac{1}{2}}\log \ro 
+ O(\ro^{-\sigma-\frac{1}{2}}\log \ro),  \\ 
& \int_{\ro}^{\infty}\sin ((2 \ro)^{\frac{1}{2}} - (2s)^{\frac{1}{2}})s^{-\sigma} ds  
= - \sqrt{2} \ro^{-\sigma + \frac{1}{2}} + O(\ro^{-\sigma-\frac{1}{2}}). 
\end{align}
\end{lem}

Let $g_{7}(\rho, \eta) = - \frac{3}{16 \ro^{2}} \eta$. 
In order to prove Theorem \ref{thm-singular-disk}, 
we seek a solution to the following problem:
\EQ{ \label{eq-eta-infty}
\begin{cases}
\frac{d^{2} \eta}{d \ro^{2}}
+ (\frac{1}{2\ro} + \frac{3}{16 \ro^{2}})\eta + f(\rho) + 
\sum_{i=1}^{7}g_{i}(\rho, \eta) - \frac{e^{-\rho} y}{4} = 0, & \qquad \rho \gg 1, \\
\lim_{\ro \to \infty} 
\eta(\ro) = 0. & 
\end{cases}
} 
To this end, we prepare several estimates, which are needed 
later. 
First, note that  
\begin{equation} \label{estimate1:phi}
\phi(\rho) = \rho^{\frac{1}{2}} + O(\rho^{-\frac{1}{2}} \log \rho). 
\end{equation}
By \eqref{estimate1:phi} and elementary calculations, 
we can obtain the following:
\begin{lem} \label{estimate-f}
Let $\phi$ be the function given by \eqref{def-phi}. 
Then, for sufficiently large $\ro>0$, we have
\begin{align}
& \phi^{-3} - \rho^{-3/2} 
= 3 \rho^{-5/2} \log \rho + O(\rho^{-7/2}\log^2\rho), \\
& \biggl| \ro^{-\frac{3}{2}}
\left\{(1- 2 \frac{\log \ro}{\ro})^{\frac{1}{2}} - 1\right\} \biggl| 
= -\rho^{-5/2}\log\rho + O(\rho^{-7/2}\log^2\rho), \\
& \phi^{-1}\rho^{-2} = \rho^{-5/2} + O(\rho^{-7/2} \log\rho), \\
& \phi^{-3} (\rho^{-1} - \rho^{-2}) = \rho^{-5/2} + O(\rho^{-7/2}\log\rho).
\end{align}
\end{lem}
\begin{lem} \label{estimate-diffe-g}
Suppose that $|\eta_{1}(\ro)|, |\eta_{2}(\ro)| \lesssim 
\ro^{-\frac{3}{2}}\log \ro$.  
Then, we have
\begin{align}
& |g_{1}(\ro, \eta_{1}) - g_{1}(\ro, \eta_{2})|, 
|g_{7}(\ro, \eta_{1}) - g_{7}(\ro, \eta_{2})| 
\lesssim 
\ro^{-2} |\eta_{1} - \eta_{2}|, \\
& |g_{2}(\ro, \eta_{1}) - g_{2}(\ro, \eta_{2})|, 
|g_{4}(\ro, \eta_{1}) - g_{4}(\ro, \eta_{2})|, 
|g_{6}(\ro, \eta_{1}) - g_{6}(\ro, \eta_{2})| \notag \\ & 
\lesssim \ro^{-2} \log \ro |\eta_{1} - \eta_{2}|, \\
& 
|g_{3}(\ro, \eta_{1}) - g_{3}(\ro, \eta_{2})|, 
|g_{5}(\ro, \eta_{1}) - g_{5}(\ro, \eta_{2})|
\lesssim \ro^{-3} \log \ro |\eta_{1} - \eta_{2}|.  
\end{align} 
\end{lem}
\begin{proof}
From the definitions of $g_{i}$ and \eqref{estimate1:phi}, we have 
\begin{align}
& |g_{1}(\ro, \eta_{1}) - g_{1}(\ro, \eta_{2})|, 
|g_{7}(\ro, \eta_{1}) - g_{7}(\ro, \eta_{2})| 
\lesssim \ro^{-2} |\eta_{1} - \eta_{2}|, \\
& |g_{2}(\ro, \eta_{1}) - g_{2}(\ro, \eta_{2})| 
\lesssim \ro^{-\frac{3}{2}} |\phi - \ro^{\frac{1}{2}}| |\eta_{1} - \eta_{2}| 
 \lesssim \ro^{-2}\log \ro |\eta_{1} - \eta_{2}|, \\
& |g_{3}(\ro, \eta_{1}) - g_{3}(\ro, \eta_{2})| 
\lesssim \ro^{-\frac{3}{2}} |\eta_{1}^{2} - \eta_{2}^{2}| 
= \ro^{-\frac{3}{2}}|\eta_{1} + \eta_{2}||\eta_{1} - \eta_{2}| \\
& \hspace{6.45cm} \lesssim \ro^{-3} \log \ro |\eta_{1} - \eta_{2}|,  
\end{align}
Let $y_{1}(\ro) = \phi(\ro) + \eta_{1}(\ro)$ and 
$y_{2}(\ro) = \phi(\ro) + \eta_{2}(\ro)$. 
Then, we obtain 
\EQ{
\begin{split}
|g_{4}(\ro, \eta_{1}) - g_{4}(\ro, \eta_{2})| 
& \lesssim \ro^{-\frac{5}{2}} \log \ro 
|(y_{1}^{2}(\ro) - \phi^{2}(\ro)) - (y_{2}^{2}(\ro) - \phi^{2}(\ro))| \\ 
& \leq \ro^{-\frac{5}{2}} \log \ro (2\phi |\eta_{1} - \eta_{2}| 
+ |\eta_{1} + \eta_{2}||\eta_{1} - \eta_{2}|) \\
& \lesssim \ro^{-2} \log \ro|\eta_{1} - \eta_{2}|.  
\end{split}
}
Since $y_{i}^{2} - \phi^{2} = 2\phi \eta_{i} + \eta_{i}^{2}$ for 
$i = 1, 2$, 
we have 
\EQ{ \label{estimate-diffe-g-eq1} 
|y^{2} - \phi^{2}| \leq |2 \phi + \eta||\eta| 
\lesssim \rho^{-1} \log \rho\qquad \mbox{for $i= 1, 2$}. 
}
It follows that 
\EQ{
\begin{split}
|g_{5}(\ro, \eta_{1}) - g_{5}(\ro, \eta_{2})|
& \lesssim \ro^{-2} |(y_{1}^{2} - \phi^{2})\eta_{1} 
- (y_{2}^{2} - \phi^{2})\eta_{2}| \\
& \lesssim \ro^{-2}|(y_{1}^{2} - \phi^{2})(\eta_{1} - \eta_{2})| 
+ \ro^{-2} |y_{1}^{2} - y_{2}^{2}||\eta_{2}| \\
& \lesssim \ro^{-3}\log \ro |\eta_{1} - \eta_{2}| 
+ \ro^{-2}\ro^{-\frac{3}{2}} (\log \ro) \ro^{\frac{1}{2}} |\eta_{1} - \eta_{2}| \\
& \lesssim 
\ro^{-3} \log \ro|\eta_{1} - \eta_{2}|. 
\end{split}
}
\EQ{
\begin{split}
& \quad |g_{6}(\ro, \eta_{1}) - g_{6}(\ro, \eta_{2})| \\
& \lesssim 
\ro^{-2} |e^{y_{1}^{2} - \phi^{2}} - 1 - (y_{1}^{2} - \phi^{2})||\eta_{1} - \eta_{2}| \\
& \quad + \ro^{-2} 
|(e^{y_{1}^{2} - \phi^{2}} - 1 - (y_{1}^{2} - \phi^{2})) 
- (e^{y_{2}^{2} - \phi^{2}} - 1 - (y_{2}^{2} - \phi^{2}))|
|\ro^{\frac{1}{2}} + \eta_{2}| \\
& \lesssim 
\ro^{-2} |e^{y_{1}^{2} - \phi^{2}} - 1 - (y_{1}^{2} - \phi^{2})|
|\eta_{1} - \eta_{2}| \\
& \quad + \ro^{-\frac{3}{2}} 
|(e^{y_{1}^{2} - \phi^{2}} - 1 - (y_{1}^{2} - \phi^{2})) 
- (e^{y_{2}^{2} - \phi^{2}} - 1 - (y_{2}^{2} - \phi^{2}))| \\
& =: I + II. 
\end{split} 
}
From \eqref{estimate-diffe-g-eq1}, we have 
\EQ{
I \lesssim \ro^{-2} |y_{1}^{2} - \phi^{2}|^{2} |\eta_{1} - \eta_{2}| 
\lesssim \ro^{-4} (\log \ro)^{2} |\eta_{1} - \eta_{2}|. 
}
By the mean value theorem and \eqref{estimate-diffe-g-eq1}, we have 
\EQ{
\begin{split}
& \quad |(e^{y_{1}^{2} - \phi^{2}} - 1 - (y_{1}^{2} - \phi^{2})) 
- (e^{y_{2}^{2} - \phi^{2}} - 1 - (y_{2}^{2} - \phi^{2}))| \\
& \lesssim 
|\exp[\theta (y_{1}^{2} - \phi^{2}) + (1-\theta)(y_{2}^{2} - \phi^{2})]
- 1||y_{1}^{2} - y_{2}^{2}| \\
& \lesssim 
|\theta (y_{1}^{2} - \phi^{2}) + (1-\theta)(y_{2}^{2} - \phi^{2})|
|y_{1}^{2} - y_{2}^{2}| \\
& \lesssim \ro^{-1} \log \ro(2\phi |\eta_{1} - \eta_{2}| 
+ |\eta_{1} + \eta_{2}| |\eta_{1} - \eta_{2}|) \\
& \lesssim \ro^{-\frac{1}{2}} \log \ro|\eta_{1} - \eta_{2}|.  
\end{split}
}
This yields that 
$II \lesssim \ro^{-2} \log \ro|\eta_{1} - \eta_{2}|$. 
Thus, we see that 
\EQ{
|g_{6}(\ro, \eta_{1}) - g_{6}(\ro, \eta_{2})| \lesssim \ro^{-2} \log \ro
|\eta_{1} - \eta_{2}|.
} 
\end{proof}
We are now in a position to prove Proposition \ref{local-singular}. 
\begin{proof}[Proof of Proposition \ref{local-singular}]
We note that \eqref{eq-eta-infty} is equivalent to the following 
integral equation: 
\EQ{
\eta(\ro) = \mathcal{T}[\eta](\ro), 
}
in which 
\EQ{
\mathcal{T}[\eta](\ro) 
= \int_{\ro}^{\infty} (\ro s)^{\frac{1}{4}}\sin ((2\ro)^{\frac{1}{2}} - (2s)^{\frac{1}{2}})
F(s, \eta) ds, 
}
where 
\EQ{
F[\ro, \eta] = f(\ro) + \sum_{i=1}^{7} g_{i}(\ro, \eta) - \frac{e^{-\ro}}{4} y. 
}
Fix $\Lambda >0$ to be sufficiently large and let $X$ be a space of 
continuous functions on $[\Lambda, \infty)$ equipped with the following norm: 
\EQ{
\|\xi\| = \sup\left\{|\ro|^{\frac{3}{2}}(\log \ro)^{-1} |\xi(\ro)| 
\mid \ro \geq \Lambda \right\}. 
}
We fix a constant $C_{*}>0$, which is defined later and 
set 
\EQ{
\Sigma = \left\{\xi \in X \mid \|\xi\| \leq 2 C_{*} \right\}. 
}
First, we shall show that $\mathcal{T}$ maps from $\Sigma$ 
to itself.  
It follows from Lemmas \ref{lem-sin} and \ref{estimate-f} 
that there exists a constant $C_{*}>0$ such that 
\EQ{ \label{eq-contra-1}
\biggl| \int_{\ro}^{\infty} (\ro s)^{\frac{1}{4}}\sin ((2\ro)^{\frac{1}{2}} - (2s)^{\frac{1}{2}})f(s) ds\biggl|
\leq C_{*} \ro^{-\frac{3}{2}} \log \ro. 
}
From Lemma \ref{estimate-diffe-g}, we have 
\EQ{
\biggl|\sum_{i=1}^{7}g_{i}(\ro, \eta)\biggl| \lesssim 
\rho^{-2} \log \ro |\eta| \lesssim \ro^{-\frac{7}{2}}(\log \ro)^{2},  
}
which yields that 
\EQ{
\begin{split}
& \biggl| \int_{\ro}^{\infty} 
(\ro s)^{\frac{1}{4}}\sin ((2\ro)^{\frac{1}{2}} - (2s)^{\frac{1}{2}})
\sum_{i=1}^{7}g_{i}(\ro, \eta) ds \biggl| 
\lesssim \int_{\ro}^{\infty} \ro^{\frac{1}{4}} s^{\frac{1}{4}} 
s^{-\frac{7}{2}} (\log s)^{2} ds \\
& \lesssim \ro^{\frac{1}{4}} \int_{\ro}^{\infty} s^{-3} ds 
\lesssim \ro^{-\frac{7}{4}}. 
\end{split}
}
Therefore, we can take $\Lambda>0$ sufficiently large 
so that 
\EQ{\label{eq-contra-2}
\biggl| \int_{\ro}^{\infty} (\ro s)^{\frac{1}{4}}\sin ((2\ro)^{\frac{1}{2}} - (2s)^{\frac{1}{2}})
\sum_{i=1}^{7}g_{i}(\ro, \eta) ds \biggl| \leq 
C_{*} \ro^{-\frac{3}{2}} \log \ro 
}
for $\ro \geq \Lambda$. 
We can easily find that 
\EQ{\label{eq-contra-2-0}
\biggl| \int_{\ro}^{\infty} 
(\ro s)^{\frac{1}{4}}\sin ((2\ro)^{\frac{1}{2}} - (2s)^{\frac{1}{2}})
e^{- s}y(s) ds \biggl| \lesssim e^{-\ro/2}. 
}
By \eqref{eq-contra-1}, \eqref{eq-contra-2} 
and \eqref{eq-contra-2-0}, 
we obtain 
\EQ{
|\mathcal{T}[\eta](\ro)| \leq 2C_{*} \ro^{-\frac{3}{2}} \log \ro
}
for $\eta \in \Sigma$. 
This yields that $\mathcal{T}[\eta] \in \Sigma$. 

Next, we shall show that 
$\mathcal{T}$ is a contraction mapping. 
For $\eta_{1}, \eta_{2} \in \Sigma$, we have 
\EQ{
\begin{split}
|\mathcal{T}[\eta_{1}](\ro) - \mathcal{T}[\eta_{2}](\ro)| 
& \leq \sum_{i=1}^{7} \int_{\ro}^{\infty} 
|(\ro s)^{\frac{1}{4}}\sin ((2\ro)^{\frac{1}{2}} - (2s)^{\frac{1}{2}})|
|g_{i}(s, \eta_{1}) - g_{i}(s, \eta_{2})| ds \\
& \quad + \int_{\ro}^{\infty} 
|(\ro s)^{\frac{1}{4}}\sin ((2\ro)^{\frac{1}{2}} - (2s)^{\frac{1}{2}})|
e^{-s}|\eta_{1} - \eta_{2}| ds. 
\end{split}
}
It follows from Lemma \ref{estimate-diffe-g} that 
\EQ{
\begin{split}
|\mathcal{T}[\eta_{1}](\ro) - \mathcal{T}[\eta_{2}](\ro)|
& \lesssim 
\int_{\ro}^{\infty} 
|(\ro s)^{\frac{1}{4}}\sin ((2\ro)^{\frac{1}{2}} - (2s)^{\frac{1}{2}})| 
s^{-2} \log s|\eta_{1} - \eta_{2}| ds\\
& + \|\eta_{1} - \eta_{2}\|e^{-\rho/2} \\
& \lesssim 
\int_{\ro}^{\infty} s^{\frac{1}{4}} \ro^{\frac{1}{4}}s^{-\frac{7}{2}}(\log s)^{2}
\|\eta_{1} - \eta_{2}\| ds + \|\eta_{1} - \eta_{2}\|e^{-\ro}\\
& \lesssim 
\|\eta_{1} - \eta_{2}\| \ro^{-2} (\log \ro)^{2}.  
\end{split}
}
This yields that 
\EQ{
\|\mathcal{T}[\eta_{1}](\ro) - \mathcal{T}[\eta_{2}](\ro)\|
\lesssim \ro^{-\frac{1}{4}} \|\eta_{1} - \eta_{2}\| < \frac{1}{2}
\|\eta_{1} - \eta_{2}\|
}
for sufficiently large $\rho > 0$. 
Thus, we find that $\mathcal{T}$ is a contraction mapping. 
This completes the proof. 
\end{proof}

\begin{proof}[Proof of Theorem \ref{thm-singular-disk}]
Let $u_{\infty}(r) = y_{\infty}(\ro)$, where $y_{\infty}(\ro)$ 
is the solution to \eqref{ODE} obtained by Proposition 
\ref{local-singular}. 
We note that $u_{\infty}$ satisfies the following:
\EQ{\label{eq-expo-rad}
- \frac{d^{2} u_{\infty}}{d r^{2}} - \frac{1}{r} \frac{d u_{\infty}}{d r}
= u_{\infty} (e^{u_{\infty}^{2}} - 1)\qquad \mbox{for 
$r \in (0, R_{\infty})$}, 
} 
where $R_{\infty} = e^{- \Lambda_{\infty}/2}$. 
Since $u_{\infty}$ is a solution to the ordinally differential 
equation \eqref{eq-expo-rad}, we can extend $u_{\infty}$ 
in the positive direction of $r$ as long as $u_{\infty}$ remains 
bounded. 
We claim that $u_{\infty}$ has a zero at some point. 
Suppose the contrary that $u_{\infty}(r) >0$ for all $0 < r < \infty$. 
Then, we see that $u_{\infty}$ is monotone decreasing.
Indeed, if not, there exists a local minimum point $r_{*} \in (0,
\infty)$. It follows that
$\partial^{2}_{r} u_{\infty}(r_{*}) \geq 0$ and
$\partial_{r} u_{\infty}(r_{*}) = 0$.
Then, from the equation \eqref{eq-expo-rad}, we obtain
\EQ{
0 \leq \frac{d^{2} u_{\infty}}{d r^{2}}(r_{*})
= - u_{\infty}(r_{*}) (e^{u_{\infty}^{2}(r_{*})} - 1) < 0,
}
which is a contradiction.
\par
Since $u_{\infty}$ is positive and monotone decreasing,
there exists a constant $C_{\infty} \geq 0$ such that
$u_{\infty}(r) \to C_{\infty}$ as $r \to \infty$. 
Suppose the contrary that $C_{\infty} > 0$. 
This together with \eqref{eq-expo-rad} yields that
\EQ{
0 = \lim_{r \to \infty} 
(\frac{d^{2} u_{\infty}}{d r^{2}}(r) + \frac{1}{r} \partial_{r} u_{\infty} (r)) 
=  - \lim_{r \to \infty}
(e^{u_{\infty}^{2}(r)} - 1)u_{\infty}(r) < 0,
}
which is absurd.
Therefore, we see that $C_{\infty} = 0$, that is, 
$\lim_{r \to \infty}u_{\infty}(r) = 0$.  
Multiplying \eqref{eq-expo-rad} by $r$ and integrating the resulting equation 
from $0$ to $r$ yields that 
\EQ{
-r \frac{d u_{\infty}}{d r}(r) 
= \int_{0}^{r} s u_{\infty} (e^{u_{\infty}^{2}} - 1) ds \geq 0. 
}
This yields that for any $R>0$, 
there exists a constant $C_{1} >0$ such that 
$- d u_{\infty}/dr(r) \geq C_{1}/r$ for all $r > R$. 
It follows that 
\EQ{
u_{\infty}(r) - u_{\infty}(R)
= \int_{R}^{r} \frac{d u_{\infty}}{d s}(s) ds \leq - C_{1}
\int_{R}^{r} \frac{1}{s} ds. 
}
Letting $r \to \infty$, we have 
\EQ{
- u_{\infty}(R)
= \lim_{r \to \infty}(u_{\infty}(r) - u_{\infty}(R)) 
\leq - C_{\infty} \lim_{r \to \infty} \int_{R}^{r} \frac{1}{s} ds
= -\infty, 
}
which is a contradiction. 
Therefore, there exists $R_{\infty} > 0$ such that 
$u_{\infty}(r)$ has a zero at $r = R_{\infty}$. 
This completes the proof.
\end{proof}
\subsection{Singular soliton by the shooting}
Let $R\in(0,R_\I)$ be the unique point such that $u_\I(R)=2$. 
For each $m\ge 0$, we put 
\EQ{ \label{expo-nonlinear}
f_{m}(s) := s(e^{s^2}-1)-m\chi(s)s, 
} 
where $\chi\in C^\I(\R)$ is a cut-off function 
satisfying $\chi(t)=1$ for $|t|\le 1, \chi(t)=0$ for $|t|\ge 2$, and $t\chi'(t)\le 0$ for all $t\in\R$. 
Consider a family of radial ODE's with a parameter $m\ge 0$:
\EQ{ \label{eq-expo-m}
 \CAS{- \fy_{*}^{\prime \prime}
- \frac{\fy_{*}^{\prime}}{r} = f_m(\fy_{*}), &(r>R)\\ 
\fy_{*}(R)=u_\I(R)=2,\ \fy_{*}^{\prime}(R)= u_\I^{\prime}(R),}}
where the prime mark denotes the differentiation with respect to $r$. 
Let $\phi_m$ be the unique solution of the above. 
We shall show that there exists $m_{*}>0$ such that 
$\phi_{m_{*}}(r)\searrow 0$ as $r\to\I$. 
To this end, we show the following:
\begin{prop} \label{thm-small-large}
Let $m\ge 0$ and $\phi_{m}$ be the solution to \eqref{eq-expo-m}.
There exists $m_{S}>0$ and $m_{L}>0$ such that 
if $m \in [0, m_{S})$, 
$\phi_{m}$ has a zero in $(R, \infty)$,  
and if $m \in (m_{L}, \infty)$, 
$\phi_{m}(r)$ is positive for all $r \ge R$. 
\end{prop}
First, we define an energy function $E_{m}:[R,\I)\to\R$ by 
\EQ{
E_{m}(r) := \frac{(\phi_{m}^{\prime}(r))^{2}}{2}  + F_m(\phi_m(r)),}
where $F_m$ is the nonlinear potential energy defined by 
\EQ{ 
 F_m(u):=\int_0^u f_m(s)ds= \frac{e^{u^2} - 1 - u^2}{2} - m \int_{0}^{u} \chi(s) s ds,}
which enjoys the standard superquadratic condition:
\EQ{ \label{SQC}
 \pt 0\le G_m(u):=uF_m'(u)-2F_m(u)=\sum_{k=2}^\I\frac{(k-1)}{k!}u^{2k} 
+ 2 m\int_0^u(\chi(s)-\chi(u))sds,
 \pr G_m(u)=0 \iff u=0,}
thanks to the monotonicity of $\chi$. 
It follows from \eqref{eq-expo-m} that  
\EQ{ \label{monotone-energy}
E_{m}^{\prime}(r) = - \frac{(\phi_{m}^{\prime}(r))^{2}}{r} \le 0. 
}
Thus, $E_{m}(r)$ is a non-increasing function of $r$. 
Using \eqref{monotone-energy}, we shall show the following:
\begin{lem} \label{positive-dec}
Let $m>0$ and $\phi_{m}$ be a solution to \eqref{eq-expo-m}. 
Let $s\ge R$, $E_{m}(s) <0$ and $\phi_{m}(s)>0$.  
Then, we have $\phi_{m}(r) >0$ for all $r \in (s, \infty)$. 
\end{lem}
\begin{proof}
Suppose the contrary. Then there exists $z \in (s, \infty)$ such that 
$\phi_{m}(z) = 0$. 
This together with \eqref{monotone-energy} yields that 
\EQ{
0>  E_{m}(s) \geq E_{m}(z) = \frac{(\phi_{m}'(z))^{2}}{2} \ge 0,  
}
which is a contradiction. 
\end{proof} 

We are now in a position to prove Proposition \ref{thm-small-large}. 
\begin{proof}[Proof of Proposition \ref{thm-small-large}]
We note that $m = 0$, $\phi_{0}(r) = u_{\infty}$ has a zero 
at $r = R_{\infty}$. 
Then, by the continuity of $\phi_m$ for $m$, 
we see that $\phi_{m}(r)$ also has a zero if $m>0$ is sufficiently small. 

On the other hand, we have 
\EQ{
 E_m(R)=\frac{(u_\I'(R))^2}{2}+\frac{e^4-5}{2}-m\int_0^2\chi(s)sds < 0}
for large $m>0$, then Lemma \ref{positive-dec} implies that $\phi_m(r)>0$ for all $r>R$. 
\end{proof}

We put 
\EQ{ \label{inf-m}
m_{*} = \inf\left\{m >0 \mid \phi_{m} (r) >0 \quad \mbox{on $r > R$ 
} \right\}. 
}
We see from Proposition \ref{thm-small-large}
that $0 < m_{*} < \infty$. 
Moreover, we have the following: 
\begin{thm} \label{exist-sing-entire}
Let $m_{*} >0$ be the number defined by \eqref{inf-m}. 
Then, $\phi_{m_{*}}$ is a singular positive radial solution to 
the following elliptic equation
\begin{equation} \label{el-expo}
\begin{cases}
- \Delta \phi = f_{m_{*}}(\phi) & \qquad \mbox{in $\R^{2}$}, \\ 
\lim_{x \to \infty}\phi(x) = 0, &
\end{cases}
\end{equation}
where $f_{m}(s)$ is defined by \eqref{expo-nonlinear}. 
Moreover, $\phi_{m_{*}}$ is strictly decreasing in the radial direction. 
Moreover, for any $m\in(0,m_*)$, there exists $C_m\in(0,\I)$ such that 
\EQ{ \label{expo-decay}
\phi_{m_{*}}(r)+|\phi_{m_{*}}^{\prime}(r)| \le C_me^{-\sqrt{m} r} \qquad 
\mbox{for all $r \ge R$}. 
}
\end{thm}
\begin{proof}
To prove Theorem \ref{exist-sing-entire}, it suffices 
to show that $\phi_{m_{*}}(r) > 0, \phi_{m_{*}}^{\prime}(r) < 0$ 
for all $r > 0$, $\lim_{r \to \infty} \phi_{m_{*}}(r) = 0$ and \eqref{expo-decay}. 

\par
First, we shall show that $\phi_{m_{*}}(r) >0$ on $r \in (R, \infty)$. 
By definition of $m_*$, there exists a sequence 
$m_{n} \searrow m_{*}$ such that 
$\phi_{m_{n}}(r) >0$ for all $r> R$ and $n \in \mathbb{N}$. 
Then $\phi_{m_*}(r)=\lim_{n\to\I}\phi_{m_n}(r) \ge 0$ for all $r> R$. 
If $\phi_{m_*}(r)=0$ at some $r>R$, then $\phi_{m_*}'(r)=0$ and so $\phi_{m_*}\equiv 0$ by the ODE, a contradiction. Hence $\phi_{m_*}(r)>0$ for all $r>R$. 

Next, we claim that 
\EQ{ \label{posi-energy-m}
E_{m_{*}}(r) \geq 0 \qquad \mbox{for all $r \in (R, \infty)$}. 
}
Suppose the contrary that there exists $R_{*}>0$ such that 
$E_{m_{*}}(R_{*}) < 0$. 
Then the continuity for $m$ yields that 
$\phi_m(r)>0$ on $R\le r\le R_*$ and $E_{m}(R_{*}) < 0$ when $m\in(0,m_*)$ is close enough to $m_*$. 
Then Lemma \ref{positive-dec} implies $\phi_m(r)>0$ for $r\ge R_*$, hence for all $r>R$,  
contradicting the definition of $m_*$. Hence we have \eqref{posi-energy-m}. 

Next, we shall show that $\phi_{m_{*}}^{\prime}(r) < 0$ for all 
$r > R$. 
Suppose the contrary and let $s>R$ be the first zero of $\phi_{m_*}'$. 
Then we have $0=\phi_{m_*}'(s)\le\phi_{m_*}''(s)=-f_{m_*}(\phi_{m_*}(s))$, 
$0\le E_{m_*}(s)=F_{m_*}(\phi_{m_*}(s))$, and so $G_{m_*}(\phi_{m_*}(s))\le 0$, contradicting \eqref{SQC}. 

Therefore $\phi_{m_*}'(r)<0<\phi_{m_*}(r)$ for all $r>R$, so $\phi_{m_*}(r)\searrow \exists C_*\in[0,2)$ and $\phi_{m_*}'(r)\to 0$ as $r\to\I$. 
Then we have $0\le\lim_{r\to\I}E_{m_*}(r)=F_{m_*}(C_*)$ and 
\EQ{
 \phi_{m_*}''(r)=-\phi_{m_*}'(r)/r-f_{m_*}(\phi_{m_*}(r))\to -f_{m_*}(C_*) \pq(r\to\I),} 
which has to be $0$ because $\phi_{m_*}'(r)\to 0$. Hence $G_{m_*}(C_*)\le 0$ and so $C_*=0$ by \eqref{SQC}. 

Finally, let $m\in(0,m_*)$. Then $\phi_{m_*}(r)\searrow 0$ together with the definition of $f_m$ implies 
\EQ{
 \phi_{m_*}''(r) = -\phi_{m_*}'(r)/r - f_{m_*}(\phi_{m_*}(r)) > m \phi_{m_*}(r) > 0}
for sufficiently large $r>R$. Hence $e^{\sqrt{m}r}(\sqrt{m}-\p_r)\phi_{m_*}(r)$ is decreasing for large $r$, which implies the desired exponential decay. 
\end{proof}
\begin{rem}
Let $\eta(\rho)$ be the solution to \eqref{eq-eta-infty}, 
obtained in the proof of Proposition \ref{local-singular}. 
Since $|\eta(\rho)| \lesssim \rho^{-3/2} \log \rho$, 
we see that $\phi_{m_{*}}$ satisfies
\EQ{ \label{sing-phi}
\pt 
\phi_{m_{*}}(r) 
= \left(-2 \log r -2 \log (-\log r) -2 \log 2 \right)^{\frac{1}{2}} 
 + O((-\log r)^{- \frac{3}{2}} \log (-\log r)) \quad 
\mbox{as $r \to 0$.}  
}
Moreover, we have, by \eqref{eq-eta-infty}, 
Lemmas \ref{lem-sin} and \ref{estimate-f}, that 
$|d^{2} \eta/d \rho^{2} (\rho)| \lesssim \rho^{-5/2} \log \rho$. 
Thus, by integrating, we see that 
$|d \eta/d \rho (\rho)| \lesssim \rho^{-3/2} \log \rho$.  
Thus, $\phi_{m_{*}}^{\prime}$ satisfies 
\EQ{\label{sing-phi-d}
\pt \phi_{m_{*}}^{\prime}(r) 
= - \left(-2 \log r -2 \log (-\log r) -2 \log 2\right)^{-\frac{1}{2}}
\left(\frac{1}{r} + \frac{1}{r\log r} \right) 
\pr \hspace{2.05cm}
 + O((-\log r)^{- 3/2} \log (-\log r)) 
\quad \mbox{as $r \to 0$.}
}
 
\end{rem}

\section{Regular solution by the heat iteration}
\label{regular-sol}
In what follows, we denote $\phi_{m_{*}}$, which is obtained in Theorem 
\ref{exist-sing-entire}, by $\fy_{*}$. 
\begin{thm} \label{regular-thm}
Let $u_0:=e^{t\De}\fy_{*}$. 
Then for any $t>0$, $u_0(t)$ is bounded on $\R^2$. 
Moreover, there exists a small $T>0$, and a solution $u_{R}$ to 
\eqref{eq-Heat-exp-m} 
with $u_{R}|_{t=0} = \varphi_{*}$ satisfying 
\EQ{
u_{R} - u_0 \in |\log t|^{-1/2}L^\I([0, T) \times \R^{2}).}
\end{thm}

Note that $\fy_{*}\in C^\I(\R^2\setminus\{0\})$ is a positive radial function 
satisfying the asymptotic form, 
\EQ{ \label{asymptotic-phi}
 \fy_{*}(r) \pt= (\ro - 2\log \ro)^{1/2}+O(\ro^{-3/2}\log\ro) 
 \pr=(\ro-2\log \ro+O(\ro^{-1}\log\ro))^{1/2} \pq (\ro\to\I),}
where we set $\ro:=|\log r^2|=2|\log r|\gg 1$.
The above two equivalent expressions of remainder will be frequently switched in the following computations. 

\par
We need precise estimate or asymptotic behavior around $t,r\to 0$ of the iteration sequence. 
Consider the first (or zeroth) iteration 
\EQ{
 u_0:=e^{t\De}\fy_{*} = \frac{1}{4\pi t}\int_{\R^2}e^{-\frac{|x-y|^2}{4t}}\fy_{*}(y)dy
 = \int_{\R^2} \frac{e^{-|z|^2/4}}{4\pi}\fy_{*}(x-\sqrt{t}z)dz.}
We shall show the following: 
\begin{lem} \label{first-iteration}
There exists $\varepsilon>0$ such that if $\max\{t, |x|^{2}\} < \varepsilon^{2}$, 
we have 
\EQ{ \label{est-u_0-claim}
u_{0}(t, x) \leq \min\{\varphi_{*}(\sqrt{t}), \varphi_{*}(x)\} + O(|\log t|^{-\frac{1}{2}}). 
}
\end{lem}

\begin{proof}
We take $\varepsilon >0$ sufficiently small so that 
for any $r < \varepsilon$, 
$\fy_{*}(r)$ satisfies \eqref{asymptotic-phi}, and 
consider the region $(t, x) \in (0, \infty) \times \R^{2}$ 
satisfying $\max\{t, |x|^{2}\} < \varepsilon^{2}$. Put
\EQ{
	\ell:=-\log t=|\log t| \gg 1,\pq \nu:=|x|^2/t\in[0,\I).}
We shall estimate $u_{0}$ by dividing 
the space-time region $\max\{t, |x|^{2}\} < \varepsilon^{2}$ 
into the three subregions: inside the parabolic cylinder $\nu\lesssim 1$, $\nu > 8 \log \ell \gg 1$ and  $1\ll \nu \leq 8 \log \ell $. 
\par
First, we consider the region $\nu \lesssim 1$. 
It follows from Lemma \ref{monotonicity} that  
\EQ{
	\|u_0\|_{L^\I_x} \pt= u_0(t, 0) = \int_{\R^2} \frac{e^{-|z|^2/4}}{4\pi}\fy_{*}(\sqrt{t}z)dz. 
}
Thus, by \eqref{asymptotic-phi}, 
we obtain 
\EQ{\label{est-u_0-13}
	\|u_0\|_{L^\I_x} 
	\pt = \int_{\R^2} \frac{e^{-|z|^2/4}}{4\pi}\fy_{*}(\sqrt{t}z)dz
	\pr =\int_0^{t} e^{-s/4}(-\log(ts) - 2\log|\log(ts)| + O(1))^{1/2}ds/4
	\pr + \int_{t}^{\frac{\varepsilon^{2}}{t}} e^{-s/4}(-\log(ts) - 2\log|\log(ts)| 
	+ O(1))^{1/2}ds/4
	\pr + \int_{\frac{\varepsilon^2}{t}}^{\infty} e^{- \frac{s}{4}} 
	\varphi_{*}(\sqrt{ts})ds/4  
	\pr =: {\rm I} + {\rm J} + {\rm K}.   
}
For $0 < s < t$, we have $ts>s^{2}$. 
This yields via an integration by parts  
\EQ{ \label{estimate-I}
	{\rm I} \lesssim 
	\int_0^{t} |2 \log s|^{1/2}ds  \lec t \ell^{1/2} 
	\lec \ell^{-1/2}.}
For $t < s < \frac{\varepsilon^2}{t}$, 
we have $|\log s| <\ell$, so the integrand is expanded 
\EQ{
	\pn(-\log(ts) - 2\log|\log(ts)|+O(1))^{1/2}
	\pt=(\ell-2\log\ell - \log s  +O(1))^{1/2}
	\pr=(\ell-2\log\ell)^{1/2} + O(\ell^{-1/2}\LR{\log s})}
Since $\varphi_{*}(r)$ is monotone decreasing in $r > 0$, 
we have 
\EQ{
	\varphi_{*}(\sqrt{ts}) \leq (\ell-2\log \ell)^{1/2} + O(\ell^{-1/2}\LR{\log s})
}
for $s \geq \frac{\varepsilon^{2}}{t}$. 
Integration against $e^{-s/4}$ yields 
\EQ{ \label{estimate-J}
	{\rm J} + {\rm K} \leq \int_{t}^{\infty} e^{-s/4} 
	((\ell-2\log \ell)^{1/2} + O(\ell^{-1/2}\LR{\log s})ds/4
	\leq (\ell-2\log \ell)^{1/2} + O(\ell^{-1/2}). 
}
The estimates \eqref{estimate-I} and \eqref{estimate-J} together 
with \eqref{est-u_0-13} imply  
\EQ{
	u_0(t, 0)-(\ell-2\log\ell)^{1/2} \lec \ell^{-1/2}.}
Namely, we have 
\EQ{ \label{est-u_{0}-s}
	u_{0} \leq \fy_{*}(\sqrt{t}) + O(\ell^{-1/2}). 
}
This is enough in the region $\nu\lec 1$.

Before focusing on the two remaining regions, 
let $0<\de\leq\varepsilon$ be a small parameter and set
\EQ{
 B:=\{z\in \R^2 \mid |x-\sqrt{t}z|<\de|x|\}.}

Decomposing $u_{0}$ as
\EQ{ \label{est-u_0-10}
u_{0}(t, x)
= \int_{B} \frac{e^{-|z|^2/4}}{4\pi}\fy_{*}(x-\sqrt{t}z)dz 
+ \int_{B^c} \frac{e^{-|z|^2/4}}{4\pi}\fy_{*}(x-\sqrt{t}z)dz 
=: u_{0}^{I} + u_{0}^{X},
}
and writing 
\EQ{\label{est-u_0-11-1}
	u_0^X \pt =\int_{B^C}\frac{e^{-|z|^2/4}}{4\pi}\large[\fy_*(x) + \fy_{*}(x-\sqrt{t}z)-\fy_*(x)\large]dz,}
one can apply the mean value theorem, 
and use \eqref{sing-phi-d} to write
\EQ{\label{estimate-fy-z-2}
	|\fy_{*}(|x-\sqrt{t}z|)-\fy_{*}(|x|)| 
	\lec \sup_{\delta|x|\leq|y|\leq|x|}|\fy_{*}^{\prime}(|y|)| \big||x - \sqrt{t}z| - |x|\big| 
	\lesssim \frac{\sqrt{t}|z|}{r|\log r|^{1/2}},
}
and then conclude that 
\EQ{\label{est-u_0-11}
	u_0^X \pt 	\pn = \int_{B^C}\frac{e^{-|z|^2/4}}{4\pi}\BR{\fy_{*}(x)+O(\frac{\sqrt{t}|z|}{r|\log r|^{1/2}})}dz
	\pr=\fy_{*}(x)(1-\te_B) + O(\sqrt{t/r^2}|\log r|^{-1/2}),}
where we set
\EQ{ 
	\te_B:=\int_B \frac{e^{-|z|^2/4}}{4\pi}dz \in (0,1).}
%
%
Now, for $z\in B$, it follows that 
\EQ{ \label{est-u_0}
|\sqrt{t} z| \geq |x| - |x - \sqrt{t}z| \geq (1 - \delta)|x|.
}
Thus, we have $|z|^2 \geq \nu (1- \delta)^{2}$, and so
\EQ{\label{est-u_0-1}
 u_0^I =\int_B \frac{e^{-|z|^2/4}}{4\pi}\fy_{*}(x-\sqrt{t}z)dz
 \le \frac{e^{-\nu(1-\de)^{2}/4 }}{4\pi t} 
\int_{|y|^2<t\de^2\nu} \fy_{*}(y)dy.}

\par
Second, we estimate $u_{0}^I$ in the region $\nu > 8 \log \ell \gg 1$.  
For $0<a\ll 1$ uniformly we have 
\EQ{\label{est-u_0-2}
 \pt\frac{1}{\pi} \int_{|y|^2<a}\fy_{*}(y)dy
 = \int_0^a (-\log s - 2\log|\log s| + O(|\log s|^{-1}\log|\log s|))^{1/2}ds 
 \pr= \int_{|\log a|}^\I (t-2\log t +O(t^{-1}\log t))^{1/2} e^{-t}dt
 \pr=a\BR{(|\log a|-2\log|\log a|)^{1/2}+\frac12|\log a|^{-1/2}+O(|\log a|^{-3/2}\log|\log a|)}
 \pr=a\BR{|\log a|-2\log|\log a|+1+O(|\log a|^{-1}\log|\log a|)}^{1/2},}
by partial integration on $e^{-t}$. 
Note that $t \de^{2} \nu = \delta |x|^{2} \leq \delta \varepsilon^{2} \ll 1$. 
Plugging the formula \eqref{est-u_0-2} 
into \eqref{est-u_0-1} yields
\EQ{\label{est-u_0-3}
 u_0^I \lec e^{-\nu (1-\de)^{2}/4}\de^2\nu(|\log(t\de^2\nu)|-2\log|\log(t\de^2\nu)|+O(1))^{1/2}.}
In the region $\nu>8\log\ell\gg 1$, 
we have $\ell = |\log t| < e^{\nu/8}$, yielding 
\EQ{
|\log(t\de^2\nu)|-2\log|\log(t\de^2\nu)| + O(1)
\pt \lesssim |\log(t\de^2\nu)| 
\pr \leq |\log t| + 2|\log \de| + |\log \nu|
\pr \leq e^{\nu/8} + 2 |\log \de| + |\log \nu| 
\pr \lesssim e^{\nu/8} + 2 |\log \de|.  
}
By \eqref{est-u_0-3}, we get the first intermediate estimate
$u_0^I \lec \nu e^{- \nu/6}$. 
Moreover, using the fact that $ |\log t| < e^{\nu/8}$ again, 
we obtain
\EQ{
|2 \log r| = - \log r^{2} = - \log (\nu t) 
= - \log \nu + |\log t| \lesssim e^{\nu/8},
}
and therefore, 
\EQ{\label{est-u_0-12}
  u_0^I \ll |\log r|^{-1/2}\nu^{-1/2}. 
}
Finally, observe that in the same region $\nu>8\log\ell\gg 1$, we have  $r^{2} \gg t$ and 
\EQ{\label{est-u_0-15}
r^{2} |\log r| = -\frac{r^{2} \log r^{2}}{2} \gg - \frac{t \log t}{2} 
= \frac{t \ell}{2}.  
}
This together with \eqref{est-u_0-10}, 
\eqref{est-u_0-11} and \eqref{est-u_0-12} yields that 
for $\nu>8\log\ell\gg 1$, we have obtained 
\EQ{ \label{est-u_{0}-l}
 u_0 - \fy_{*}(x) \lec  |\log r|^{-1/2}\nu^{-1/2} \ll \ell^{-1/2}.
}

\par 
Thus it remains to estimate $u_0^I$ in the region 
\EQ{
 1 \ll \nu:=\frac{|x|^2}{t} < 8\log\ell.}
Let $\z:=z-x/\sqrt{t}$. On $z\in B$, we have $|\z|<\de\sqrt{\nu}\ll\ell$. 
Thus, by \eqref{asymptotic-phi}, the integrand is expanded 
\EQ{ \label{estimate-fy-3}
 \pt\fy_{*}(x-\sqrt{t}z)=\fy_{*}(\sqrt{t}\z)
 \pr=(-\log t - 2\log|\log t|-\log|\z|^2-2\log|1-\ell^{-1}\log|\z|^2|+O(1))^{1/2}
 \pr=(\ell - 2\log\ell + O(\langle \log|\zeta| \rangle))^{1/2}
 \pr=(\ell - 2\log\ell)^{1/2} + O(\ell^{-1/2} \langle \log |\zeta| \rangle).}
Note that the above $O$ terms can be bigger than the main part, and indeed unbounded as $\zeta\to 0$. The last expansion is nevertheless valid, using that the $O$ term on the 3rd line is positive. 
We have
\EQ{
\int_{B} e^{- \frac{|z|^{2}}{4}} \langle \log |\zeta| \rangle dz
\leq \int_{|\z| \leq \delta \sqrt{\nu}}  \langle \log |\zeta| \rangle d \z 
\lesssim \delta^{2} \nu \langle \log \delta^{2} \nu \rangle 
\lesssim 1.  
}
This together with \eqref{estimate-fy-3} yields that 
\EQ{ \label{est-u_0-14}
 u_0^I 
= \int_{B} \frac{e^{-|z|^{2}}/4}{4\pi} \fy_{*}(x - \sqrt{t}z) dz 
\pt = (\ell-2\log\ell)^{1/2} \int_B\frac{e^{-|z|^2/4}}{4\pi} dz +O(\ell^{-1/2}) 
\pr \leq \fy_{*}(\sqrt{t})\int_B\frac{e^{-|z|^2/4}}{4\pi} dz +O(\ell^{-1/2}).}
Hence, by \eqref{est-u_0-10}, \eqref{est-u_0-11} 
and \eqref{est-u_0-14}, 
we have   
\EQ{
 u_0 \le \fy_{*}(x) + (\fy_{*}(\sqrt{t})-\fy_{*}(x))\int_B \frac{e^{-|z|^2/4}}{4\pi} dz + O(\ell^{-1/2}),}
where the remainder on $B$ is estimated by 
\EQ{
 \pt \fy_{*}(\sqrt{t})-\fy_{*}(x) \lec \frac{-\log\sqrt{t}+\log r}{\ell^{1/2}} \lec \frac{\log\nu}{\ell^{1/2}},
 \pr \int_B e^{-|z|^2/4}dz \leq e^{-\nu(1-\de)^{2}/4}\nu < e^{-\nu/8}.}
Here, we have used \eqref{est-u_0} in the second inequality. 
This yields that 
\EQ{\label{est-u_{0}-m}
u_{0} \le \fy_{*}(x) + O(\ell^{-\frac{1}{2}}). }
From \eqref{est-u_{0}-l}, \eqref{est-u_{0}-s} and 
\eqref{est-u_{0}-m}, we see that \eqref{est-u_0-claim} 
holds.  
\end{proof}

From Lemma \ref{first-iteration}, we have obtained, denoting $\ro:=|\log r^2|$,  
\EQ{ \label{est u0}
 \max(t, |x|^2) <\e^{2} \pt\implies 
u_0 \le \min(\fy_{*}(\sqrt{t}),\fy_{*}(x))+O(\ell^{-1/2})
 \pr\implies u_0^2 \le \min(\ell - 2\log\ell,\ro-2\log\ro)+O(1)
 \pr\implies e^{u_0^2} \lec \min(\frac{1}{t\ell^2},\frac{1}{r^2\ro^2})
 \pr\implies |u_0|e^{u_0^2} \lec \min(\frac{1}{t\ell^{3/2}},\frac{1}{r^2\ro^{3/2}}), \pq u_0^2e^{u_0^2} \lec \min(\frac{1}{t\ell},\frac{1}{r^2\ro}).}
One can use the radial monotonicity of $u_0$ (cf.~Lemma \ref{monotonicity}) 
to extend, to all $x\in\R^2$,  the bounds of 
the functions $u_0e^{u_0^2}$ and $u_0^2e^{u_0^2}$ 
as follows: 
\EQ{ \label{est-F}
 t<\e^{2} \implies u_0 \le \sqrt{\ell}, \pq \pt u_0e^{u_0^2} \lec [(t+r^2)^{-1}+\e^{-2}]|\log \min\{t + r^{2}, \e^{2}\}|^{-3/2}=:F_0,
 \pr u_0^2e^{u_0^2} \lec [(t+r^2)^{-1}+\e^{-2}]|\log \min\{t + r^{2}, \e^{2}\}|^{-1}=:F_0'.}
Hence by the mean value theorem, for any functions $v_0$ and $v_1$, we have, 
for $t<\e$, that  
\EQ{\label{est-F-diff}
 \sqrt{\ell}(|v_0|+|v_1|)\lec 1 \implies |f_{0}(u_0+v_0)-f_{0}(u_0+v_1)| \pt \lec u_0^2e^{u_0^2}|v_0-v_1| 
 \pr\lec F_0'|v_0-v_1|, } 
where $f_{0}(u) = u (e^{u^{2}} -1)$. 

\par
To estimate the second iteration, we prepare the following:
\begin{lem} \label{integral-formula}
Let $\alpha > 0$ and $0 < \e < 1$. 
For any $(t, r) \in (0, \e^{2}) \times (0, \infty)$, 
there exists a positive constant $C_{*}$ such that 
\EQ{
\label{outer to iter1}
 \int_0^t e^{(t-s)\De}\frac{1}{s+r^2}|\log \min\{s + r^{2}, \e^{2}\}|^{- \alpha}ds 
\leq C_{*} \ell^{-\alpha},  
}
\EQ{
\label{outer to iter2}
 \int_0^t e^{(t-s)\De}|\log \min\{s + r^{2}, \e^{2}\}|^{- \alpha}ds 
\leq C_{*} t \ell^{-\alpha}
}
\end{lem}
\begin{proof}
From Lemma \ref{monotonicity}, we have
\EQ{
\pt \int_0^t e^{(t-s)\De}\frac{1}{s+r^2}|\log \min\{s + r^{2}, \e^{2}\}|^{- \alpha} ds 
\pr \leq \int_{0}^{t} \int_{0}^{\infty} 
\frac{re^{- \frac{r^{2}}{4(t-s)}}}{2(t-s)(s + r^{2})}
|\log \min\{s + r^{2}, \e^{2}\}|^{- \alpha} 
dr ds.  
}
Then, the integral is estimated using the following formula
\EQ{ 
 \pt \int_0^t e^{(t-s)\De}\frac{1}{s+r^2} |\log \min\{s + r^{2}, \e^{2}\}|^{- \alpha} ds 
 \pr\le\int_0^t \int_0^\I\frac{e^{-\frac{\s}{4s}}}{4s(t-s+\s)}
|\log \min\{t - s + \s, \e^{2}\}|^{- \alpha} d\s ds
 \pr=\int_0^\I \int_0^t \frac{e^{-\y/4}}{t-s+s\y} 
|\log \min\{t -s+\y, \e^{2}\}|^{- \alpha}ds \frac{d\y}{4}
}
where the variables are changed by $s\to t-s$, $r^2\to\s\to s\y$. 
Then, we split the integral as follows. 
\EQ{
 \pt \int_0^\I \int_0^t \frac{e^{-\y/4}}{t-s+s\y} 
|\log \min\{t-s+s\y, \e^{2}\}|^{- \alpha}ds \frac{d\y}{4}
 \pr =  \int_0^{1/\sqrt{t}}\int_0^{t} \frac{e^{-\y/4}}{t-s+s\y} 
|\log \min\{t-s+s\y, \e^{2}\}|^{- \alpha}ds \frac{d\y}{4}
 \pr + \int_{1/\sqrt{t}}^{\I} \int_0^t \frac{e^{-\y/4}}{t-s+s\y} 
|\log \min\{t-s+s\y, \e^{2}\}|^{- \alpha}ds \frac{d\y}{4}
 =: I + II. 
}

We first estimate $I$.   
Since $t -s + s \y \leq \sqrt{t}$ for $0 < s < t$ and $0 < \y < 1/\sqrt{t}$, 
we have 
\EQ{ \label{est-int I}
I 
\pt \leq \int_0^{1/\sqrt{t}}\int_0^{t} \frac{e^{-\y/4}}{t-s+s\y} 
|\log \min\{\sqrt{t}, \e^{2}\}|^{- \alpha}ds \frac{d\y}{4}
\pr \lesssim \ell^{-\alpha} \int_{0}^{1/\sqrt{t}} \int_{0}^{t} 
\frac{e^{-\y/4}}{t-s+s\y} ds d\y
= \ell^{-\alpha} \int_{0}^{1/\sqrt{t}} e^{-\y/4}\frac{\log \y}{\y -1} d\y 
\lesssim \ell^{-\alpha}. 
}
Next, we estimate $II$. 
We note that $|\log \min\{s + r^{2}, \e^{2}\}|^{- \alpha} 
\leq (- \log \e^{2})^{-\alpha}$ and $t -s + s\y \geq t$ 
for $0 < s < t$ and $\y \geq 1/\sqrt{t}$. 
Therefore, we have 
\EQ{ \label{est-int II}
II 
\lesssim (- \log \e^{2})^{-\alpha} 
\int_{1/\sqrt{t}}^{\I} e^{- \y/4}\int_0^t \frac{1}{t} 
ds d\y \lesssim e^{- \frac{1}{4\sqrt{t}}}. 
}
From \eqref{est-int I} and \eqref{est-int II}, 
we obtain \eqref{outer to iter1}. 
W can obtain \eqref{outer to iter2} by the similar argument above. 
This completes the proof. 
\end{proof}
Using Lemma \ref{integral-formula}, we shall show the following:
\begin{lem} \label{estimate-D}
Let $\e>0$ be given by Lemma \ref{integral-formula}. 
For any space-time function $v$ on $(0,\e^{2})\times \R^2$, let
\EQ{
 D[v]:=\int_0^t e^{(t-s)\De}f_0((u_0+v)(s))ds.}
Then, there exists a positive constant $C_{0}$ such that  
\EQ{ \label{Duh iter1-1}
 \pt |D[0]| \le C_{0} \ell^{-3/2}.   
} 
Moreover, for any $C_1\in(0,\I)$, there exists $C_2\in(0,\I)$ such that for any $v_0$ and $v_1$ satisfying  
\EQ{
 \sup_{0<t<\e}|\log t|^{1/2}\|v_j(t)\|_{L^{\infty}} \le C_1 \pq (j=0,1), }
and for any $\alpha \ge -1$, we have
\EQ{ \label{Duh diff-1}
 |D[v_0]-D[v_1]| 
\le C_2 \ell^{-1-\al}\sup_{0<s<t}|\log s|^\al\|v_0(s)-v_1(s)\|_\I.  
}
\end{lem}
\begin{proof}
\eqref{est-F} and Lemma \ref{integral-formula} yield
\EQ{ \label{Duh iter1}
 |D[0]| \pt\le \int_0^t e^{(t-s)\De}|f_0(u_0(s))|ds
  \lec \int_0^t e^{(t-s)\De}F_0(s)ds
 \pr\lesssim 
\ell^{-3/2} + \varepsilon^{-2} t \ell^{-3/2} 
\lec \ell^{-3/2}.}  
Similarly, \eqref{est-F-diff} and Lemma \ref{integral-formula} yield
\EQ{
 \pt|D[v_0]-D[v_1]| \le \int_0^t e^{(t-s)\De}|f_0(u_0+v_0))-f_0(u_0-v_0)|ds
  \pr\lec \int_0^t e^{(t-s)\De}F_0'(s)\|v_1(s)-v_0(s)\|_\I ds
  \pr\lec \int_0^t e^{(t-s)\De}[(s+r^2)^{-1}+\e^{-2}]
|\log \min\{s + r^{2}, \e^{2}\}|^{-1-\al}
\sup_{0<s<t}|\log s|^\al\|v_0(s)-v_1(s)\|_\I
  \pr\lec \ell^{-1-\al}\sup_{0<s<t}|\log s|^\al\|v_0(s)-v_1(s)\|_\I,}
for any $\al\ge-1$. 
Noting that the implicit constants are independent of $v_0,v_1,\al,t$, we arrive at the desired conclusion. 
\end{proof}
We are now in a position to prove 
Theorem \ref{regular-thm}
\begin{proof}[Proof of Theorem \ref{regular-thm}]
We put  
\EQ{
E[v] := \int_{0}^{t} e^{(t-s)\Delta} L(u_{0} + v)(s)ds, \qquad 
L(u) := m_{*} \chi(u)u 
}
and $I[v] := D[v] - E[v]$. 
Then, we are naturally lead to consider the mapping $v\mapsto I[v]$ for $v$ 
in the following set 
\EQ{
 B_T^{1/2}:=\{v \in C([0,T]\times \R^2) \mid \|v\|_{X_T^{1/2}}:=\sup_{0<t<T}\ell^{1/2}\|v(t)\|_\I\le 1\},}
for some constant $T\in(0,\e)$ to be determined, which is a closed ball of a Banach space with the norm $X_T^{1/2}$. The estimates on $D[\cdot]$ in \eqref{Duh iter1-1} and \eqref{Duh diff-1} with $C_1=1$ imply 
\EQ{
 \pt \|D[v_0]-D[v_1]\|_{X_T^{1/2}} \le C_2|\log T|^{-1}\|v_0-v_1\|_{X_T^{1/2}},
 \pr \|D[v_0]\|_{X_T^{1/2}} \le \|D[0]\|_{X_T^{1/2}}+\|D[0]-D[v_0]\|_{X_T^{1/2}} \le (C_0+C_2)|\log T|^{-1}, 
}
for any $v_0,v_1\in B_T^{1/2}$. 
Moreover, we can easily obtain  
\EQ{
 \pt \|E[v]\|_\I \le t\|L\|_\I,
 \pr \|\int_0^t e^{(t-s)\De}(L(v_0)-L(v_1))(s)ds\|_\I
 \le t\|L\|_{Lip}\|v_0-v_1\|_\I. 
} 
Hence if $T>0$ is small enough then $v\mapsto I[v]$ is a contraction mapping on $B_T^{1/2}$, 
so there is a unique fixed point $v\in B_T^{1/2}$. 
Then $u=u_0+v$ is a local (strong) solution on $0<t<T$ of 
\EQ{
 \dot{u} - \De u + m_{*} \chi(u) u = f_0(u), \pq u(0)=\fy_{*}.} 
Lemma \ref{estimate-D} also implies that $v\in X_T^{3/2}$. 
This completes the proof. 
\end{proof}
We can prove Theorem \ref{main-thm} immediately from 
Theorems \ref{exist-sing-entire} and \ref{regular-thm}. 
\begin{proof}[Proof of Theorem \ref{main-thm}]
We also denote $\phi_{m_{*}}$, which is the stationary singular soliton to 
\eqref{eq-Heat-exp-m} obtained in Theorem 
\ref{exist-sing-entire}, by $\fy_{*}$. 
Let $u_{S}(t) = \varphi_{*}$ and $u_{R}(t)$ be the regular solution to 
\eqref{eq-Heat-exp-m}, obtained in Theorem \ref{regular-thm}. 
We see that $u_{S}(0) = u_{R}(0) = \varphi_{*}$. 
However, $u_{S}(t) = \varphi_{*} \notin L^{\infty}
(\mathbb{R}^{2})$ while $u_{R}(t) \in L^{\infty}(\mathbb{R}^{2})$ for all $t \in (0, T)$. 
This implies that $u_{S}(t) \neq u_{R}(t)$ for all $t \in (0, T)$. 
This completes the proof. 
\end{proof}

\appendix
\section{Maximum point of solutions to the linear Heat equation}
In this appendix, we shall give a proof of the fact, which is used 
in Section 3. 
More precisely, we show the following: 
\begin{lem} \label{monotonicity}
Let $\phi$ be a radially decreasing function. 
Set $u(t) = e^{t \Delta} \phi$. Then, 
$u(t)$ is also radially decreasing and  
\EQ{
\|u(t, \cdot)\|_{L_{x}^{\infty}} = u(t, 0). 
}
\end{lem}
\begin{proof}
Note that $u$ is also radial. 
Setting $v = \partial_{r}u$, we see that 
$v$ satisfies the following: 
\EQ{ \label{eq-V}
\dot{v} - \Delta v + \frac{1}{r^{2}} v = 0. 
}
We put $v_{+} = \max\{v, 0\}$. 
Multiplying \eqref{eq-V} by $v_{+}$ and integrating the resulting 
equation over $\R^{2}$, we have 
\EQ{ \label{monotonicity-V}
\partial_{t} \|v_{+}\|_{L^{2}}^{2} = 
-\|\nabla v_{+}\|_{L^{2}}^{2} - \int \frac{1}{r^{2}} |v_{+}|^{2} dx \leq 0. 
}
From the assumption, we infer that $v_{+}(0) = 0$. 
This together with \eqref{monotonicity-V} yields that 
$v_{+} \equiv 0$ in $(0, \infty) \times \R^{2}$. 
This completes the proof. 
\end{proof}

\subsection*{Acknowledgement}
This work was done while H.K. was visiting at University of Victoria. 
H.K. thanks all members of the Department of 
Mathematics and Statistics for their warm hospitality. 
The work of S.I. was supported by NSERC grant (371637-2014).
The work of H.K. was supported by 
JSPS KAKENHI Grant Number JP17K14223. 
K.N. was supported by JSPS KAKENHI Grant Number JP17H02854.
The research of J.W. is partially supported by NSERC of Canada.

\vspace{0.5cm}

\noindent
Slim Ibrahim 
\\
Department of Mathematics and Statistics
\\
University of Victoria 
\\
3800 Finnerty Road, Victoria, B.C., Canada V8P 5C2
\\
E-mail: ibrahims@uvic.ca

\vspace{0.5cm}

\noindent
Hiroaki Kikuchi
\\
Department of Mathematics
\\
Tsuda University
\\
2-1-1 Tsuda-machi, Kodaira-shi, Tokyo 187-8577, JAPAN
\\
E-mail: hiroaki@tsuda.ac.jp

\vspace{0.5cm}

\noindent
Kenji Nakanishi 
\\
Research Institute for Mathematical Sciences
\\
Kyoto University 
\\
Oiwake Kita-Shirakawa, Sakyo Kyoto 606-8502 JAPAN
\\
E-mail: kenji@kurims.kyoto-u.ac.jp

\vspace{0.5cm}

\noindent
Juncheng Wei
\\
Department of Mathematics
\\
University of British Columbia
\\
1984 Mathematics Road
Vancouver, B.C., Canada V6T 1Z2
\\
E-mail:  jcwei@math.ubc.ca

\end{document}